\documentclass[12pt,reqno]{amsart}
\usepackage[utf8]{inputenc}
\usepackage[english]{babel}
\usepackage{amsmath}
\usepackage{amsfonts}
\usepackage{amssymb}
\usepackage{amsthm}
\usepackage{hyperref}
\usepackage{cleveref}
\usepackage{color}

\makeatletter
\newcounter{savesection}
\newcounter{apdxsection}
\renewcommand\appendix{\par
	\setcounter{savesection}{\value{section}}%
	\setcounter{section}{\value{apdxsection}}%
	\setcounter{subsection}{0}
	\gdef\thesection{Appendix \@Alph\c@section}}

\newcommand\unappendix{\par
	\setcounter{apdxsection}{\value{section}}%
	\setcounter{section}{\value{savesection}}%
	\setcounter{subsection}{0}%
	\gdef\thesection{\@arabic\c@section}}
\makeatother

\newtheorem*{theorem*}{Theorem}
\newtheorem{theorem}{Theorem}[section]

\newtheorem{lemma}[theorem]{Lemma}
\newtheorem{remark}[theorem]{Remark}
\newtheorem*{remark*}{Remark}
\newtheorem*{lemma*}{Lemma}
\newtheorem{example}[theorem]{Example}
\newtheorem*{example*}{Example}

\newcommand{\norm}[1]{\left\lVert#1\right\rVert}

\begin{document}
\title[Composite Quadrature Methods for Conv. Integrals]{Composite Quadrature Methods for Weakly Singular Convolution Integrals}
\author[Wesley Davis]{W. Davis}
\address{Department of Mathematics and Statistics, Old Dominion University, Norfolk, VA.}
\email{wdavi002@odu.edu}
\author[Richard Noren]{R. Noren}
\address{Department of Mathematics and Statistics, Old Dominion University, Norfolk, VA.}
\email{rnoren@odu.edu}

\begin{abstract}
The well-known Caputo fractional derivative and the corresponding Caputo fractional integral occur naturally in many equations that model physical phenomena under inhomogeneous media. The relationship between the two fractional terms can be readily obtained by applying the Laplace transform to a given equation. We seek to numerically approximate Caputo fractional integrals using a Taylor series expansion for convolution integrals. This naturally extends into being able to approximate convolution integrals for a wider class of convolution integral kernels $K(t-s)$. One of the main advantages under this approach is the ability to numerically approximate weakly singular kernels, which fail to converge under traditional quadrature methods. We provide stability and convergence analysis for these composite quadratures, which offer optimal convergence for approximating functions in $C^{\gamma}[0,T]$, where $\alpha \leq \gamma \leq 5$ and $0<\alpha < 1$. For the order $\gamma = 1,2,3,4,5$ scheme, the resulting approximation is $O(\tau^{\gamma})$ accurate, where $\tau$ is the size of the partition of the time domain. By instead utilizing a fractional Taylor series expansion, we are able to obtain for $\gamma \in (0,5)-\{1,2,3,4\}$ order scheme, which yields an approximation of $O(\tau^{\gamma})$ with a constant dependent on the kernel function which improves the order of convergence. This allows for a far wider class of functions to be approximated, and by strengthening the regularity assumption, we are able to obtain more accurate results. General convolution integrals exhibit these results up to $\gamma = 2$ without the assumption of $K$ being decreasing. Finally, some numerical examples are presented, which validate our findings.  
\end{abstract}
\maketitle
\section{Introduction}
We begin by defining the Caputo fractional time-derivative [9,10] of a given function $f(t)$ as
$$^{C}_{0}D_t^{\alpha}f(t) = \dfrac{1}{\Gamma(1-\alpha)}\int_0^t \dfrac{df(s)}{ds}(t-s)^{-\alpha}\, ds,\quad 0<\alpha <1$$
which is a fractional derivative of order $\alpha$.
In \cite{DavisX}, the use of the Laplace transform was applied to a Caputo fractional derivative term to obtain the fractional integral term
\begin{equation}\label{CapFI}
^{C}_{0}I^{\alpha}_t f(t)  = \dfrac{1}{\Gamma(\alpha)}\int_0^t (t-s)^{\alpha-1}f(s) \, ds,
\end{equation}
which was studied numerically and convergent schemes were developed for this integral inspired by the works presented in [9-13], [15-16]. The integral \eqref{CapFI} can be expressed as the convolution $a_{\alpha}*f$, where $a_{\alpha}(t)= t^{\alpha-1}$. 
Our work will examine and derive numerical schemes to discretize integrals of the form \eqref{CapFI}, which has numerous engineering and physics applications, see \cite{Stynes2017}, \cite{Sun2006} and \cite{Zhang2014}. One of the major advantages of applying the Laplace Transform to a fractional derivative term, as seen in \cite{DavisX}, is the ability to preserve the same assumption of regularity as required for fractional derivative discretizations and recovering an extra $O(\tau_k^{\alpha})$ order of accuracy, where $\tau_k$ denotes the time step size. Further, we now have the ability to relax the regularity assumption from requiring the objective function $f(t)\in C^2[0,T]$ under the usual L1-method to instead be any $f(t) \in C^{\gamma}[0,T]$, where $\alpha \leq \gamma \leq 5$ and $0< \alpha < 1$. This is achieved by a usual Taylor series expansion to obtain convergence results for whole number values of $\gamma$, and by utilizing a fractional Taylor series expansion to approximate functions with a fractional order of regularity, see \cite{Usero2008}. By requiring more regularity, we are able to obtain a higher order of convergence, as seen in Theorems 3.6 and 3.7. This method naturally generalizes to any convolution type-quadrature where the kernel function $K$ is positive, decreasing, and satisfies $K\in L^1[0,T]$, as seen in Theorems 3.4 and 3.5.
\par The remainder of the paper is organized as follows. Section 2 will provide discretizations for fractional integrals of the above form, and a general scheme is established for fractional integrals of other forms based on the integral kernel. We obtain general schemes of orders up to $5^{th}$ order of accuracy. Section 3 establishes all of the necessary consistency, stability, and convergence results for each of these schemes, in addition to a discussion of the implementation of the schemes. We prove optimal order of convergence of our stable schemes, where the order is at least 3. The instability of schemes of order 6 and above are presented as well. The main convergence results are featured in Theorems 3.4 through 3.6. Section 4 presents two fractional integral equations as numerical examples that validate our findings. Future works will consider the application of these methods to fractional order diffusion processes based on the validity of the schemes and the order of accuracy they can provide.
\section{Discretized numerical schemes}
In order to discretize the Caputo fractional integral \eqref{CapFI}, we must consider the Taylor expansion for a given function $f(t)$ at the arbitrary point $s\in [0,T)$, the usual time interval over which the integral is considered. That is, the Taylor expansion centered at the point $t_k$ for $f(s)$ on $[0,T)$, given $s\in [0,T)$, is 
\begin{equation}
f(s) = f(t_k) + (s-t_k)f'(t_k)+\dfrac{(s-t_k)^2}{2!}f''(t_k) + \dfrac{(s-t_k)^3}{3!}f'''(t_k)+... 
\end{equation} 
Define $\tau_k = t_k - t_{k-1}$ and let $t_0 = 0$ and $t_N = T$ such that $\{t_k\}_{k=0}^{N} \in [0,T]$. From the above, similar Taylor expansions centered at any given $t_k$ can be constructed for each of the previous points $t_{k-1},t_{k-2},...,t_1,t_0\in [0,T-\tau_k]$. That is, 
\begin{align}
f(t_k) &= f(t_k)\\
f(t_{k-1}) &= f(t_k)-\tau_k f'(t_k) + \dfrac{\tau_k^2}{2!}f''(t_k)- \dfrac{\tau_k^3}{3!}f'''(t_k) + O(\tau_k^4)\\
f(t_{k-2}) &= f(t_k)-2\tau_k f'(t_k) + \dfrac{(2\tau_k)^2}{2!}f''(t_k)- \dfrac{(2\tau_k)^3}{3!}f'''(t_k) + O(\tau_k^4)\\
f(t_{k-3}) &= f(t_k)-3\tau_k f'(t_k) + \dfrac{(3\tau_k)^2}{2!}f''(t_k)- \dfrac{(3\tau_k)^3}{3!}f'''(t_k) + O(\tau_k^4)\\
&...\nonumber\\
f(t_0) &= f(t_k)-k\tau_k f'(t_k) + \dfrac{(k\tau_k)^2}{2!}f''(t_k)- \dfrac{(k\tau_k)^3}{3!}f'''(t_k) + O(\tau_k^4)
\end{align}
Thus, the $j$-th order approximation of $f(s)$ for any $s\in[0,T)$ at the point $t_k \in [0,T)$ requires we solve the linear system of equations:
\begin{align}
\sum_{i=0}^{j-1} c^k_i f(t_{k-i}) &= f(s)\label{s1}\\ 
&= \sum_{i=0}^{j-1} \dfrac{(s-t_k)^i}{i!}f^{(i)}(t_k) + O((s-t_k)^j).\label{s2}
\end{align}
Each of the terms $f(t_{k-i})$ are then replaced by their Taylor expansions about the point $t_k$ and then solved, after neglecting the higher order terms of $O(\tau_k^j)$ and $O((s-t_k)^j)$. For example, a second order approximation of $f(s)$ is provided in \cite{DavisX}, Theorem 3.2, which can be recovered by solving the equation
\begin{align*}
c^k_0 f(t_k)+ c^k_1\left(f(t_k)-\tau_k f'(t_k) \right) &= f(t_k) + (s-t_k)f'(t_k).
\end{align*}
This equation can be rewritten as a system of equations
\begin{align*}
c^k_0 + c^k_1 &= 1\\
-c^k_1 \tau_k &= (s-t_k),
\end{align*}
Which yields the solution $c^k_1 = \dfrac{t_k-s}{\tau_k}$, $c^k_0 = 1-\dfrac{t_k-s}{\tau_k}$. Therefore, we may numerically approximate the integral as seen in \cite{DavisX} using $c^k_0$ and $c^k_1$ as solved for above:
\begin{align*}
\int_0^{t_n} \dfrac{(t_n-s)^{\alpha-1}}{\Gamma(\alpha)}f(s)\, ds &= \sum_{k=1}^{n} \int_{t_{k-1}}^{t_{k}} \dfrac{(t_n-s)^{\alpha-1}}{\Gamma(\alpha)}f(s)\, ds\\
&\approx \sum_{k=1}^{n} \int_{t_{k-1}}^{t_{k}} \dfrac{(t_n-s)^{\alpha-1}}{\Gamma(\alpha)}\left(c^k_0f(t_k)+c^k_1f(t_{k-1})\right)\, ds,
\end{align*}
which recovers the equation that was studied in greater detail in \cite{DavisX}. We remark that under this construction, we satisfy the condition $s\in [t_{k-1},t_k]$. This directly implies that the coefficients $c_0$ and $c_1$ presented above are positive. We now provide the values of the coefficients for each scheme up to 4th order accuracy. Higher order schemes can be derived using the generalized system of equations  \eqref{s1} and \eqref{s2}. We remark that in general, $c_i = c_i(s)$ for each i.\\
\\
\textbf{First order accurate:}
\begin{align*}
c^k_0&=1,\\
f(s) &= f(t_k) + O(\tau_k).
\end{align*} 
\\
\textbf{Second order accurate:}
\begin{align*}
c^k_0 &= 1-\dfrac{t_k-s}{\tau_k},\, c^k_1 = \dfrac{t_k-s}{\tau_k},\\
f(s) &= \left(1-\dfrac{t_k-s}{\tau_k}\right)f(t_k) +  \left(\dfrac{t_k-s}{\tau_k}\right)f(t_{k-1}) +O(\tau_k^2).
\end{align*}
\\
\textbf{Third order accurate:}
\begin{align*}
c^k_0 &= \dfrac{(\tau_k+s-t_k)(2\tau_k+s-t_k)}{2\tau_k^2},\, c^k_1 = \dfrac{(t_k-s)(2\tau_k+s-t_k)}{\tau_k^2},\\ \, c^k_2 &= \dfrac{(s-t_k)(\tau_k+s-t_k)}{2\tau_k^2}\\
f(s) &= \left( \dfrac{(\tau_k+s-t_k)(2\tau_k+s-t_k)}{2\tau_k^2}\right)f(t_k) + \left(\dfrac{(t_k-s)(2\tau_k+s-t_k)}{\tau_k^2}\right)f(t_{k-1}) \\&+  \left(\dfrac{(s-t_k)(\tau_k+s-t_k)}{2\tau_k^2}\right)f(t_{k-2})+ O(\tau_k^3).
\end{align*}
\\
\textbf{Fourth order accurate:}
\begin{align*}
c^k_0 &= \dfrac{(\tau_k+s-t_k)(2\tau_k+s-t_k)(3\tau_k+s-t_k)}{6\tau_k^3},\\ c^k_1 &= \dfrac{(t_k-s)(2\tau_k+s-t_k)(3\tau_k+s-t_k)}{2\tau_k^3},\\ \, c^k_2 &= \dfrac{(s-t_k)(\tau_k+s-t_k)(3\tau_k+s-t_k)}{2\tau_k^3},\\ c^k_3 &= \dfrac{(t_k-s)(\tau_k+s-t_k)(2\tau_k+s-t_k)}{6\tau_k^3}\\
f(s) &= \left(\dfrac{(\tau_k+s-t_k)(2\tau_k+s-t_k)(3\tau_k+s-t_k)}{6\tau_k^3} \right)f(t_k) \\&+ \left( \dfrac{(t_k-s)(2\tau_k+s-t_k)(3\tau_k+s-t_k)}{2\tau_k^3}\right)f(t_{k-1}) \\&+  \left(\dfrac{(s-t_k)(\tau_k+s-t_k)(3\tau_k+s-t_k)}{2\tau_k^3}\right)f(t_{k-2})\\&+\left(\dfrac{(t_k-s)(\tau_k+s-t_k)(2\tau_k+s-t_k)}{6\tau_k^3}\right)f(t_{k-3}) + O(\tau_k^4).
\end{align*}
We present the generalized equation to solve for a n-th order accurate approximation to $f(s)$. The above equations may be rewritten by the matrix equation
\begin{equation}\label{ME1}
\begin{bmatrix}
1 & 1 & 1 & ... & 1\\
0 & -\tau_k & -2\tau_k & ... & -(n-1)\tau_k\\
0 & (-\tau_k)^2 & (-2\tau_k)^2 & ... & (-(n-1)\tau_k)^2\\
& &...\\
0 & (-\tau_k)^{n-1} & (-2\tau_k)^{n-1} & ... & (-(n-1)\tau_k)^{n-1}
\end{bmatrix}
\begin{bmatrix}
c^k_0\\
c^k_1\\
c^k_2\\
...\\
c^k_{n-1}
\end{bmatrix}
= 
\begin{bmatrix}
1\\
(s-t_k)\\
(s-t_k)^2\\
...\\
(s-t_k)^{n-1}
\end{bmatrix}.
\end{equation}
The first $n \times n$ matrix is, in fact, the transpose of the usual Vandermonde matrix \cite{Turner1966} where each entry, other than the first row of ones, is a multiple of the time partition $\tau_k$. That is, we set $x_1 = 0$, $x_2 = -\tau_k$, $x_3 = -2\tau_k$, ... $x_n = -(n-1)\tau_k$ so that
\begin{equation}\label{ME2}
V^T_{\tau_k}\vec{c^k_n} = \vec{y^k_n},
\end{equation}
where $V^T_{\tau_k}$ is the Vandermonde matrix with each $x_n$ expressed in terms of $\tau_k$ as above, and
$$
\vec{c^k_n} = 
\begin{bmatrix}
c^k_0\\
c^k_1\\
c^k_2\\
...\\
c^k_{n-1}
\end{bmatrix}, \,
\vec{y^k_n} = 
\begin{bmatrix}
1\\
(s-t_k)\\
(s-t_k)^2\\
...\\
(s-t_k)^{n-1}
\end{bmatrix}.
$$
The following lemma asserts the existence of any general n-th order numerical scheme. Since we have 
\begin{align*}
\det(V^T_{\tau_k})=\det(V_{\tau_k})&=\prod_{1\leq i < j \leq n}(x_j-x_i)\\
&= \prod_{1\leq i < j \leq n}(i-j)\tau_k \neq 0,
\end{align*}
provided that $\tau_k>0$ which directly implies that the matrix $V^T_{\tau_k}$ is invertible under this condition. The following lemma is immediate from the above. 
\begin{lemma}\label{L1}
Let $\tau_k>0$. Then, \eqref{ME2} has a unique solution $\vec{c^k_n}$ for each $n\leq N\in \mathbb{N}$.
\end{lemma}
We now compute the unique solution based on the previous lemma. From \cite{Turner1966}, we can establish the generalized inverse of the Vandermonde matrix. 
\begin{theorem}\label{Thm1}
Let $\tau_k>0$. Then, \eqref{ME2} has a unique solution $\vec{c^k_n}$ for each $n\leq N\in \mathbb{N}$, with the solution 
\begin{align}\label{c_j}
\left[c^k_i\right]=\begin{cases}
&\displaystyle{\sum_{1\leq j\leq n}} (s-t_k)^{j-1} (-1)^{n-i} \left(\dfrac{\displaystyle{\sum_{\substack{1\leq m_1, m_2, ... , m_{n-i},\\m_1,...,m_{n-1}\neq j}}} x_{m_1}...x_{m_{n-i}}}{\displaystyle{\prod_{1\leq i < j \leq n}\left(\tau_k (i-j)\right)}}\right):1\leq i < n \\
\\
&\displaystyle{\sum_{1\leq j\leq n} (s-t_k)^{j-1} \dfrac{1}{\displaystyle{\prod_{1\leq i < j \leq n}\left(\tau_k (i-j)\right)}} \quad \quad \quad \quad: i=n}
\end{cases}
\end{align}
\end{theorem}
\begin{proof}
From Lemma \ref{L1}, we may invert the matrix $V^T_{\tau_k}$ to obtain the solution
$$\vec{c^k_n} = (V^T_{\tau_k})^{-1}\vec{y^k_n},$$
where then from \cite{Turner1966}, each entry of $(V^T_{\tau_k})^{-1} = \left[v_{ij}\right]$, $1\leq i,j\leq n$ is calculated by
$$v_{ij} =  \begin{cases}
&(-1)^{n-i} \left(\dfrac{\displaystyle{\sum_{\substack{1\leq m_1, m_2, ... , m_{n-i}\\m_1,...,m_{n-1}\neq j}} x_{m_1}...x_{m_{n-i}}}}{\displaystyle{\prod_{1\leq i < j \leq n}\left(\tau_k (i-j)\right)}}\right):1\leq i < n \\
\\
&\dfrac{1}{\displaystyle{\prod_{1\leq i < j \leq n}\left(\tau_k (i-j)\right)}} \quad \quad \quad \quad: k=n
\end{cases}$$
so we may solve component-wise to find each entry of $\vec{c_n} =\left[c_j\right]$, $1\leq j \leq n$, from
$$
(V^T_{\tau_k})^{-1}\vec{y^k_n} = \sum_{1\leq j\leq n} v_{ij}y^k_j,
$$
where $\vec{y_n} =\left[y_j\right]$, $1\leq j \leq n$. Thus,
\begin{align}
&\left[c^k_i\right]=(V^T_{\tau_k})^{-1}\vec{y^k_n} = \sum_{j} v_{ij}y^k_j\\ 
&= \begin{cases}
&\sum_{1\leq j\leq n} (s-t_k)^{j-1} (-1)^{n-i} \left(\dfrac{\displaystyle{\sum_{\substack{1\leq m_1, m_2, ... , m_{n-i}\\m_1,...,m_{n-1}\neq j}} x_{m_1}...x_{m_{n-i}}}}{\displaystyle{\prod_{1\leq i < j \leq n}\left(\tau_k (i-j)\right)}}\right):1\leq i < n \\
\\
&\sum_{1\leq j\leq n} (s-t_k)^{j-1} \dfrac{1}{\prod_{1\leq i < j \leq n}\left(\tau_k (i-j)\right)}: \quad i=n.\label{c_i}
\end{cases}
\end{align}
\end{proof}
\begin{remark}
By utilizing the fractional Taylor series expansion instead for $f(s)$ on $[0,T)$, as discussed in \cite{Usero2008},  we may obtain similar results to those outlined in Theorem 2.2. This can further relax the regularity assumption to $f(s) \in C^{\alpha}[0,T]$.
\end{remark}
Using the fractional Taylor series expansion, we can assert that we have an $\alpha$ order scheme defined by the following:
 \\
\textbf{$\alpha$ order accurate:}
\begin{align*}
c^k_0&=1,\\
f(s) &= f(t_k) + O(\tau_k^{\alpha}).
\end{align*} 
\\
We now examine the consistency, stability, and convergence of these schemes based on the generalized scheme
\begin{equation}
f(s) = \sum_{i=0}^{n-1} c^k_i f(t_{k-i}) + O((s-t_k)^n).
\end{equation}
\section{Numerical Consistency, Stability, and Convergence}
\subsection{Numerical Consistency and Stability}
We motivate our discussion of stability by examining the results presented in \cite{Baker1993}. The quadrature studied in \cite{Baker1993} is of the form 
\begin{equation}\label{scheme}
\int_0^T \phi(s)\, ds = \tau \sum_{j=0}^N w_j \phi(j\tau) + O(\epsilon),
\end{equation}
which provides the error estimate given $\phi \in C^R[0,T]$ a sequence of constants, $\{c_l \}$, such that
$$\int_0^T \phi(s)\, ds - \tau \sum_{j=0}^N w_j \phi(j\tau) =  \sum_{l=\rho+1}^R h^l(r^l c_l)\{\phi^{(l-1)}(T)-\phi^{(l-1)}(0) \}+ O(h^R).$$
We will directly compare these results to the ones established in the previous section to prove stability and assert convergence. Our goal is to decompose the integrand into a convolution integral $\phi(s) = K(t_n - s)f(s)$, where we may relax the continuity assumptions on the kernel function $K(t_n-s)$. This goal is motivated in part from the results obtained in \cite{DavisX}, where we seek a generalization for the integral kernel. We begin by recalling some basic definitions for quadrature methods. From (1.15) of \cite{Baker1993}, a quadrature method is said to be consistent if it satisfies 
\begin{equation*}
\sum_{j=0}^{N} w_j = N
\end{equation*}
for the global quadrature of the integral \eqref{scheme}. We will relate \eqref{scheme} and a generalization of the results provided in \cite{DavisX}.
\begin{lemma}
Let $0\leq s\leq t_n$ for any prescribed $t_n \in [0,T]$. Let $\gamma$ denote the order of the desired approximation to the function $f(s)$, let $\phi(s) = f(s)K(t_n-s)$ such that $f(s) \in C^{\gamma}[0,T]$ and $K\in L^1[0,T]$ . Then, for an order $\gamma$ scheme, as described in Theorem \ref{Thm1}, we have
\begin{equation}
\int_0^{t_n} \phi(s)\, ds = \sum_{k=1}^{n}\sum_{j=0}^{\gamma-1}w_j^k f(t_{k-j}) + O(\epsilon).
\end{equation}
\end{lemma}
\begin{proof}
By utilizing the taylor expansion for $f(s)$ about the point $t_k$, we may readily obtain a similar quadrature rule by using Theorem \ref{Thm1} and the definition of each $c_j(s)$ defined in \eqref{c_j}. By further remarking that for each $s\in [t_{k-1},t_k]$, then we may write $O((s-t_k)^{\gamma}) = O(\epsilon)$. Thus,
\begin{align}
\int_0^{t_n} &\phi(s)\, ds = \int_0^{t_n} f(s)K(t_n-s)\, ds\\
&= \sum_{k=1}^n \int_{t_{k-1}}^{t_k} f(s)K(t_n-s)\, ds\\
&= \sum_{k=1}^n \int_{t_{k-1}}^{t_k} \left(\sum_{j=0}^{\gamma-1}c^k_j(s)f(t_{k-j}) + O(\epsilon) \right)K(t_n-s)\, ds\\
&= \sum_{k=1}^n \sum_{j=0}^{\gamma-1} f(t_{k-j}) \int_{t_{k-1}}^{t_k} c^k_j(s) K(t_n-s)\, ds + O(\epsilon).
\end{align} 
By letting 
\begin{equation}
w_j^k = \int_{t_{k-1}}^{t_k} c^k_j(s)K(t_n-s)\, ds,
\end{equation}
we arrive at the conclusion.
\end{proof}
The following remark is a natural extension of the first lemma, which allows for direct comparison to prove stability using the Theorem 3.7 in \cite{Baker1993}.
\begin{remark}\label{doublesum}
By expanding the series
$$ \sum_{k=1}^{n}\sum_{j=0}^{\gamma-1}w_j^k f(t_{k-j})+ O(\epsilon),$$
and by collecting all of the repeating terms for each $f(t_{k-j})$, we may condense the double summation into a single summation term
\begin{equation}
 \sum_{k=1}^{n}\sum_{j=0}^{\gamma-1}w_j^k f(t_{k-j}) = \sum_{k=0}^{n}(w_0^k + w_1^{k+1} + ... + w_{\gamma-1}^{k+\gamma-1})f(t_k),
\end{equation}
where we note that $w_0^0 = 0$ to satisfy the previous lemma. Further, by defining for fixed $\gamma>0$
\begin{equation}
\tilde{w}_k^{\gamma} = w_0^k + w_1^{k+1} + ... + w_{\gamma-1}^{k+\gamma-1},
\end{equation}
We arrive at a form identical to the generalized quadrature rule posed in \cite{Baker1993}
\begin{equation}\label{scheme2}
\sum_{k=1}^{n}\sum_{j=0}^{\gamma-1}w_j^k f(t_{k-j}) + O(\epsilon) = \sum_{k=0}^{n}\tilde{w}_k^{\gamma} f(t_k)+O(\epsilon).
\end{equation}
\end{remark}
\begin{theorem}\label{Consistency}
The approximation scheme \eqref{scheme2} is consistent for any $\gamma>0$, where $\gamma$ is the order of approximation.  
\end{theorem}
\begin{proof}
From the consistency requirement in \cite{Baker1993}, we must show that the scheme \eqref{scheme2} satisfies any time step $\tau>0$
\begin{equation}
\begin{cases}\label{sys1}
\displaystyle \int_0^{t_n} \phi(s) \,ds = \tau \sum_{j=0}^n w_{n-j}\phi(j\tau)+O(\epsilon) \\
\displaystyle \sum_{j=0}^n w_j = n
\end{cases}
\end{equation}
for any fixed $\gamma$. That is, we have from Remark \ref{doublesum}
\begin{align}
\sum_{k=0}^n \tilde{w}_k^{\gamma} &= \sum_{k=0}^n w_0^k + w_1^{k+1} + ... + w_{\gamma-1}^{k+\gamma-1}\\
&= \sum_{k=1}^{n}\sum_{j=0}^{\gamma-1}w_j^k \\
&= \sum_{k=1}^{n}\sum_{j=0}^{\gamma-1} \int_{t_{k-1}}^{t_k} c^k_j(s)K(t_n-s)\, ds\\
&= \sum_{k=1}^{n}\int_{t_{k-1}}^{t_k} \left(\sum_{j=0}^{\gamma-1} c^k_j(s)\right)K(t_n-s)\, ds.
\end{align}
From \eqref{ME2}, the first equation in the Vandermonde matrix requires $\sum_{j=0}^{\gamma-1} c^k_j(s) =1$, hence,
\begin{align}
\sum_{k=0}^n \tilde{w}_k^{\gamma} = \sum_{k=1}^{n}\int_{t_{k-1}}^{t_k} K(t_n-s)\, ds &= \int_{0}^{t_n} K(t_n-s)\, ds.
\end{align}
On the other hand, by relabelling the coefficients of \eqref{sys1} and by noting that $k\tau = t_k$,
\begin{align}
\int_0^{t_n} f(s)K(t_n-s) \,ds &= \tau \sum_{j=0}^n w_{n-j}f(j\tau)K(t_n-j\tau)+O(\epsilon)\\
&= \tau\sum_{k=0}^n w_{n-k}f(t_k)K(t_n-t_k)+O(\epsilon)\label{c2}.
\end{align}
By equating \eqref{scheme2} and \eqref{c2}, we have
\begin{align}\label{int1key}
\tau\sum_{k=0}^n w_{n-k}K(t_n-t_k) =  \sum_{k=0}^n \tilde{w}_k^{\gamma} &= \int_{0}^{t_n} K(t_n-s)\, ds.
\end{align}
Since each $w_{n-k}$ is arbitrary under this construction, we select $w_{n-k}$ to satisfy $\sum_{k=0}^n w_k = n$.
Thus, we have for the scheme \eqref{scheme2}
\begin{equation}
\begin{cases}
\displaystyle \int_0^{t_n}f(s)K(t_n-s)\, ds = \tau\sum_{k=0}^n w_{n-k}f(t_k)K(t_n-t_k)+O(\epsilon)\\
\displaystyle \sum_{k=0}^n w_k = n,
\end{cases}
\end{equation}
hence the scheme \eqref{scheme2} is consistent. For simplicity and for implementation, we take $w_k=1$ for each k to trivially satisfy these conditions since $w_0^0 = w_0 = 0$.  
\end{proof}
We must further satisfy stability requirements in order to prove the convergence of these schemes for any order $\gamma >0$. From \cite{Baker1993}, we have the following theorem asserting stability under arbitrary quadrature rules:
\begin{theorem*}\textbf{(3.7 of \cite{Baker1993})}
The stability polynomial 
\begin{align}\label{stabilitypoly}
\Sigma (\mu;\lambda \tau) = &(1-\lambda \tau w_0 K(0))\mu^N - \lambda \tau w_1 K(\tau)\mu^{N-1} - ...  \nonumber\\&- \lambda \tau w_N K(n\tau)
\end{align}
is Schur, if $\left| \lambda \tau \right| \sum_{k=0}^N \left| w_k K(k\tau) \right| <1$. Assuming each $w_k \geq 0$ and satisfy $\sum_{k=0}^N w_k = N$, the recurrence for 
$$y(t) = f(t) + \lambda\int_{t_n-T}^{t_n} K(t_n-s)y(s)\, ds$$ 
when $K(t) \equiv 1$ for $t\in [0,T]$ is stable whenever $\left| \lambda T \right| <1$, given $\tau>0$.
\end{theorem*}
We remark that under these results, we must simply satisfy the requirement that each $\tilde{w}^{\gamma}_k \geq 0$ in \eqref{scheme2} to satisfy a similar stability criterion for this generalized quadrature. This leads immediately to two stability results:
\begin{theorem}
Let $K(s)$ be a positive kernel on $[0,T]$ and let $\tau_k>0$ for all k. Then, the approximation scheme \eqref{scheme2} is stable for $\alpha \leq \gamma \leq 2$ and $0<\alpha<1$, where $\gamma$ is the order of approximation.
\end{theorem}
\begin{proof}
The case where $\gamma =1$ is immediate since $c_0^k=1$, hence $\tilde{w}_1^k\geq 0$. For $\gamma =2$, then 
\begin{align}
\tilde{w}^k_2 &= w_0^k + w_1^{k+1}\\
&= \int_{t_{k-1}}^{t_k}c_0^k(s)K(s)\, ds + \int_{t_{k}}^{t_{k+1}}c_1^{k+1}(s)K(s)\, ds \\
&\geq \min_{s\in[t_1,T]}\left|K(s)\right| \left(\int_{t_{k-1}}^{t_k}c_0^k(s)\, ds + \int_{t_{k}}^{t_{k+1}}c_1^{k+1}(s)\, ds \right)\\
&= \min_{s\in[t_1,T]}\left|K(s)\right| \left(\int_{t_{k}}^{t_k+1}c_0^{k+1}(s)\, ds + \int_{t_{k}}^{t_{k+1}}c_1^{k+1}(s)\, ds \right)\\
&= \min_{s\in[t_1,T]}\left|K(s)\right| \left(\int_{t_{k-1}}^{t_k}c_0^k(s) +c_1^{k+1}(s)\, ds \right)\\
&= \min_{s\in[t_1,T]}\left|K(s)\right| \tau_{k+1} \geq 0.
\end{align}
Using similar analysis we are able to come to the same conclusion for $\gamma = \alpha$ and $\gamma = 1+\alpha$, given $0<\alpha<1$. Therefore, when $\gamma \in [1,2]$, the scheme \eqref{scheme2} is stable.
\end{proof}
We require additional assumptions on the integral kernel $K(s)$ to ensure that the scheme is stable in the case where the order of approximation to \eqref{scheme2} is any order $2<\gamma\leq 5$. 
\begin{theorem}
Let $K(s)$ be a positive, nonincreasing kernel on $[0,T]$ and let $\tau_k>0$ for all k. The approximation scheme \eqref{scheme2} is stable for any $2< \gamma \leq 5$ order of accuracy.
\end{theorem}
\begin{proof}
We begin by showing that $\tilde{w}^{\gamma}_k \geq 0$ for each $k=1,2,...,n$. That is, we use the relationship established in Remark \ref{doublesum}. We will present the argument for the cases where $\gamma = 3,4,5$ and deduce the pattern from there. We remark that under the construction found in Theorem \ref{Thm1} that for $j=2,4,6,...$ then $c_j^k(s) < 0$, provided $s\in[t_{k-1},t_k]$. Therefore, when $\gamma = 3$, we have
\begin{align*}
\tilde{w}^k_3 &= w_0^k + w_1^{k+1}+w_2^{k+2}\\
&= \int_{t_{k-1}}^{t_k}c_0^k(s)K(s)\, ds + \int_{t_{k}}^{t_{k+1}}c_1^{k+1}(s)K(s)\, ds + \int_{t_{k+1}}^{t_{k+2}}c_2^{k+2}(s)K(s)\, ds \\
&\geq  K(t_{k+1})\left(\int_{t_{k-1}}^{t_k}c_0^k(s)\, ds + \int_{t_{k}}^{t_{k+1}}c_1^{k+1}(s)\, ds + \int_{t_{k+1}}^{t_{k+2}}c_2^{k+2}(s)\, ds \right)\\
&= K(t_{k+1})\left(\int_{t_{k+1}}^{t_{k+2}}c_0^{k+2}(s)\, ds + \int_{t_{k+1}}^{t_{k+2}}c_1^{k+2}(s)\, ds + \int_{t_{k+1}}^{t_{k+2}}c_2^{k+2}(s)\, ds \right)\\
&= K(t_{k+1})\left(\int_{t_{k+1}}^{t_{k+2}}c_0^{k+2}(s) + c_1^{k+2}(s) + c_2^{k+2}(s) \,ds \right)\\
&= K(t_{k+1})\tau_{k+2} \geq 0.
\end{align*}
Hence, when $\gamma =3$, the scheme \eqref{scheme2} is stable. When $\gamma =4$, the argument is similar, but we must account for the extra positive term in $w_3^{k+3}$. That is, utilizing a similar argument in the case where $\gamma=3$ for the first three terms,
\begin{align*}
\tilde{w}^k_4 =& w_0^k + w_1^{k+1}+w_2^{k+2}+w_3^{k+3}\\
=& \int_{t_{k-1}}^{t_k}c_0^k(s)K(s)\, ds + \int_{t_{k}}^{t_{k+1}}c_1^{k+1}(s)K(s)\, ds + \int_{t_{k+1}}^{t_{k+2}}c_2^{k+2}(s)K(s)\, ds \\
&+ \int_{t_{k+2}}^{t_{k+3}}c_3^{k+3}(s)K(s)\, ds\\
\geq& K(t_{k+1})\left( \int_{t_{k}}^{t_{k+1}}c_0^{k+1}(s) + c_1^{k+1}(s)+ c_2^{k+1}(s) \,ds\right) \\
&+ \int_{t_{k+2}}^{t_{k+3}}c_3^{k+3}(s)K(t_{k+3})\, ds\\
\geq& K(t_{k+3})\left( \int_{t_{k+2}}^{t_{k+3}}c_0^{k+3}(s) + c_1^{k+3}(s)+ c_3^{k+3}(s) \,ds\right) \\
&+ \int_{t_{k+2}}^{t_{k+3}}c_2^{k+3}(s)K(t_{k+1})\, ds\\
=& K(t_{k+3})\left( \int_{t_{k+2}}^{t_{k+3}} 1 - c_2^{k+3}(s)\, ds\right) + \int_{t_{k+2}}^{t_{k+3}}c_2^{k+3}(s)K(t_{k+1})\, ds\\
=& \int_{t_{k+2}}^{t_{k+3}} K(t_{k+3}) + (K(t_{k+1})-K(t_{k+3}))c_2^{k+3}(s)\, ds\geq 0,
\end{align*}
since K is nonincreasing, $K(t_{k+1})\geq K(t_{k+3})$, and since $c_2^{k+3}(s) <0$ where $s\in [t_{k+2},t_{k+3}]$ by translating over to the correct interval, so we require then $-1\leq c_2^{k+3}(s) < 0$, $s\in[t_{k+2},t_{k+3}]$ to finish the proof. To satisfy the requirement, we find that the minimum attained on the interval $s\in[t_{k+2},t_{k+3}]$ for the function $c_2^{k+3}(s)$ is found at $s=t_{k+3} + \dfrac{-4+\sqrt{7}}{3}\tau_{k+3}$ by the Extreme Value Theorem and by evaluating the derivative of $c_2^{k+3}(s)$ on the interval $s\in[t_{k+2},t_{k+3}]$. Hence, the minimum value for $c_2^{k+3}(s)$ is 
\begin{align}
&c_2^{k+3}(t_{k+3} + \dfrac{-4+\sqrt{7}}{3}\tau_{k+3})=\\
&\dfrac{\left(\dfrac{-4+\sqrt{7}}{3} \right) \left(1+\dfrac{-4+\sqrt{7}}{3} \right) \left(3+\dfrac{-4+\sqrt{7}}{3}\right)}{2\tau^2_{k+3}}\\
 &= \dfrac{20-14\sqrt{7}}{54}\approx -0.31 \geq -1.
\end{align}
Since this minimum value is still greater than $-1$, this scheme is stable for any choice of $\tau_k$.
Therefore, when $\gamma=4$, the scheme satisfies the condition. We finally consider the case where $\gamma=5$. In this case, we have a similar argument to the case where $\gamma=4$, but we add an additional negative term in $w_4^{k+4}(s) <0$ for $s\in [t_{k+3},t_{k+4}]$. Thus,
 \begin{align*}
\tilde{w}^k_5 =& w_0^k + w_1^{k+1}+w_2^{k+2}+w_3^{k+3}+w_4^{k+4}\\
=& \int_{t_{k-1}}^{t_k}c_0^k(s)K(s)\, ds + \int_{t_{k}}^{t_{k+1}}c_1^{k+1}(s)K(s)\, ds + \int_{t_{k+1}}^{t_{k+2}}c_2^{k+2}(s)K(s)\, ds \\
&+ \int_{t_{k+2}}^{t_{k+3}}c_3^{k+3}(s)K(s)\, ds + \int_{t_{k+3}}^{t_{k+4}}c_4^{k+4}(s)K(s)\, ds \\
\geq& K(t_{k+3})\left( \int_{t_{k+1}}^{t_{k+2}}c_0^{k+2}(s) + c_1^{k+2}(s)+ c_3^{k+2}(s) \,ds\right) \\
&+ K(t_{k+1})\left(\int_{t_{k+1}}^{t_{k+2}}c_2^{k+3}(s)+c_4^{k+2}(s)\, ds \right)\\
=& K(t_{k+3})\left( \int_{t_{k+1}}^{t_{k+2}}1-c_2^{k+2}(s)-c_4^{k+2}(s) \,ds\right) \\
&+ K(t_{k+1})\left(\int_{t_{k+1}}^{t_{k+2}}c_2^{k+3}(s)+c_4^{k+2}(s)\, ds \right)\\
=&  \int_{t_{k+1}}^{t_{k+2}} K_{t_{k+3}} + \left(K_{t_{k+1}} - K_{t_{k+3}} \right)\left(c_2^{k+2}(s) +
 c_4^{k+2}(s) \right) \, ds 
\end{align*}
Thus, we must restrict $-1 \leq c_2^{k+2}(s) + c_4^{k+2}(s) < 0$ to ensure the stability of the scheme. We remark that under the construction of the coefficients $c_2^{k+2}$ and $c_4^{k+2}$, we have the common factor of $(s-t_{k+2})$ and $(s-t_{k+2}+\tau_{k+2})$, hence $c_2^{k+2}(s) + c_4^{k+2}(s)=0$ where $s = t_{k+2}$ and $s= t_{k+2}-\tau_{k+2} = t_{k+1}$. Since $c_2^{k+2},\,c_4^{k+2} < 0$ for $s\in (t_{k+1},t_{k+2})$, then we may apply the extreme value theorem again to assert that $c_2^{k+2}(s) + c_4^{k+2}(s)$ attains a minimum value on $[t_{k+1},t_{k+2}]$. Hence by utilizing Mathematica's solution software, the minimum of $c_2^{k+2}(s) + c_4^{k+2}(s)$ is attained at $s \approx t_{k+2} - 0.416\tau_{k+2}$, with a minimum value of $$c_2^{k+2}(t_{k+2} - 0.416\tau_{k+2}) + c_4^{k+2}(t_{k+2} - 0.416\tau_{k+2}) \approx -0.603912 \geq -1,$$
hence, the scheme satisfies the condition and is therefore stable. \par
We will now show that the above condition no longer holds for schemes of order $\gamma =6$. By repeating the same argument for when $\gamma = 6$, we have 
 \begin{align*}
\tilde{w}^k_6 =& w_0^k + w_1^{k+1}+w_2^{k+2}+w_3^{k+3}+w_4^{k+4}+w_5^{k+5}\\
=& \int_{t_{k-1}}^{t_k}c_0^k(s)K(s)\, ds + \int_{t_{k}}^{t_{k+1}}c_1^{k+1}(s)K(s)\, ds + \int_{t_{k+1}}^{t_{k+2}}c_2^{k+2}(s)K(s)\, ds \\
&+ \int_{t_{k+2}}^{t_{k+3}}c_3^{k+3}(s)K(s)\, ds + \int_{t_{k+3}}^{t_{k+4}}c_4^{k+4}(s)K(s)\, ds+\int_{t_{k+4}}^{t_{k+5}}c_5^{k+5}(s)K(s)\, ds \\
\geq& K(t_{k+5})\left( \int_{t_{k+1}}^{t_{k+2}}c_0^{k+2}(s) + c_1^{k+2}(s)+ c_3^{k+2}(s)+c_5^{k+2} \,ds\right) \\
&+ K(t_{k+1})\left(\int_{t_{k+1}}^{t_{k+2}}c_2^{k+3}(s)+c_4^{k+2}(s)\, ds \right)\\
=& K(t_{k+5})\left( \int_{t_{k+1}}^{t_{k+2}}1-c_2^{k+2}(s)-c_4^{k+2}(s) \,ds\right) \\
&+ K(t_{k+1})\left(\int_{t_{k+1}}^{t_{k+2}}c_2^{k+3}(s)+c_4^{k+2}(s)\, ds \right)\\
=&  \int_{t_{k+1}}^{t_{k+2}} K_{t_{k+5}} + \left(K_{t_{k+1}} - K_{t_{k+5}} \right)\left(c_2^{k+2}(s) +
 c_4^{k+2}(s) \right) \, ds, 
\end{align*}
where we again must satisfy $-1 \leq c_2^{k+2}(s) + c_4^{k+2}(s) < 0$ to ensure the stability of the scheme. Using the same argument as before, we find that there exists a minimum for $s\in (t_{k+1},t_{k+2})$, then we may apply the extreme value theorem again to assert that $c_2^{k+2}(s) + c_4^{k+2}(s)$ attains a minimum value on $[t_{k+1},t_{k+2}]$. Using the definition of the coefficients $c_2^{k+2}$ and $c_4^{k+2}$ as defined by \eqref{c_j}, we find that the minimum exists at the point $s = t_{k+2}-0.38843\tau_{k+2}$ with the minimum value $c_2^{k+2}(s) + c_4^{k+2}(s)  = -1.05315 \ngeq -1$. A similar analysis holds for each of the fractional order schemes and is therefore omitted. Hence, the condition is no longer satisfied and thus the scheme fails to be stable for when $\gamma =6$, which completes the proof.
\end{proof}
The instability of the scheme for $\gamma \geq 6$ is unsurprising as this technique still exhibits the Runge phenomenon for interpolating polynomials at higher orders. We note that the assumption that $K(s)$ is decreasing allows for such discussion as seen in \cite{Zhang2014} and \cite{DavisX}, initially motivated by the findings in \cite{Friedman1969}. If instead $K(s)$ is assumed to be positive and nondecreasing, the scheme is trivially stable. Unfortunately, this modification is not useful as the assumption that $K(s)$ is decreasing is vital to asserting the existence and uniqueness of a solution to equations with such an integral term, namely by using fixed-point theory. With the consistency and stability results, we are now ready to present the convergence analysis. 
\subsection{Numerical Convergence}
We now consider an arbitrary stable scheme of the form \eqref{scheme2} up to order $\gamma$ where $\alpha \leq \gamma \leq 5$.  We present the convergence results for the usual Taylor series expansion first, followed by the fractional Taylor series expansion results.
\begin{theorem} 
Let $0\leq s\leq t_n$ for any prescribed $t_n \in [0,T]$. Let $K\in L^1[0,T]$ be positive and nonincreasing on $[0,T]$ and let $\tau_k>0$ for all k. Let $f(s)\in C^{\gamma}[0,T]$ satisfy the stable scheme \eqref{scheme2} up to some order $\gamma =1,2,3,4, 5$, where $\gamma$ is the order of approximation. Then, for some $C>0$,
\begin{equation}
\norm{\int_0^{t_n} f(s) K(t_n-s)\, ds -  \sum_{k=0}^{n}\tilde{w}_k^{\gamma} f(t_k)}_{\infty} \leq C\max_{1\leq k\leq n}\tau_k^{\gamma}. 
\end{equation}
\begin{proof}
We fix $\gamma \geq 1$ such that for some $C_1 >0$,
\begin{align}
\norm{\int_0^{t_n} f(s) K(t_n-s)\, ds -  \sum_{k=1}^{n}\tilde{w}_k^{\gamma} f(t_k)}_{\infty} & \\=\norm{\sum_{k=1}^{n}\int_{t_{k-1}}^{t_k}\left(f(s)-\sum_{j=0}^{\gamma-1}c_j^k(s)f(t_{k-j})\right)K(t_n-s)\, ds}_{\infty}&\\
=\norm{\sum_{k=1}^{n}\int_{t_{k-1}}^{t_k}\left(C_1(t_k-s)^{\gamma}\right)K(t_n-s)\, ds}_{\infty}  \\
\leq C_1\max_{1\leq k \leq n}\tau_k^{\gamma}\norm{\sum_{k=1}^{n}\int_{t_{k-1}}^{t_k}|K(t_n-s)|\, ds}_{\infty}& \\ 
=C_1\max_{1\leq k \leq n}\tau_k^{\gamma}\norm{\int_{0}^{t_n}|K(t_n-s)|\, ds}_{\infty}&\\ 
\leq C\max_{1\leq k \leq n}\tau_k^{\gamma},
\end{align}
where $\displaystyle{C_1 = C_0\max_{0\leq t \leq t_n}\lvert f^{(\gamma)}(s)\rvert}<\infty$ for some $C_0>0$ and $$0<  C_1\norm{K(t_n-s)}_{L^1[0,t_n]} = C < \infty.$$ 
\end{proof}
\end{theorem}
For the fractional order regularity assumption, we the following convergence rate results.
\begin{theorem} 
Let $0\leq s\leq t_n$ for any prescribed $t_n \in [0,T]$. Let $K\in L^1[0,T]$ be positive and nonincreasing on $[0,T]$ and let $\tau_k>0$ for all k. Let $f(s)\in C^{\gamma}[0,T]$ satisfy the stable scheme \eqref{scheme2} up for any $\gamma \in (0,5)-\{1,2,3,4\}$, where $\gamma$ is the order of approximation. Let $\gamma = n+\alpha$, where $n=0,1,2,3,4$ and $0<\alpha<1$. Then, for some $C>0$,
\begin{equation}
\norm{\int_0^{t_n} f(s) K(t_n-s)\, ds -  \sum_{k=0}^{n}\tilde{w}_k^{\gamma} f(t_k)}_{\infty} \leq C\max\left(\max_{1\leq k\leq n}\tau_k^{\gamma},\max_{1\leq k\leq n}\tau_k^{n+1}\right). 
\end{equation}
\begin{proof}
By fixing $\gamma = n+\alpha$ where $\gamma \in  (0,5)-\{1,2,3,4\}$ and $0<\alpha<1$ , we have for some $C_1>0$,
\begin{align}
\norm{\int_0^{t_n} f(s) K(t_n-s)\, ds -  \sum_{k=1}^{n}\tilde{w}_k^{\gamma} f(t_k)}_{\infty} & \\=\norm{\sum_{k=1}^{n}\int_{t_{k-1}}^{t_k}\left(f(s)-\sum_{j=0}^{\gamma-1}c_j^k(s)f(t_{k-j})\right)K(t_n-s)\, ds}_{\infty}&\\
\leq\norm{\sum_{k=1}^{n}\int_{t_{k-1}}^{t_k}\left(C_1\max\left(\max_{1\leq k\leq n}\tau_k^{\gamma},\max_{1\leq k\leq n}\tau_k^{n+1}\right)\right)K(t_n-s)\, ds}_{\infty}  \\
\leq C_1\max\left(\max_{1\leq k\leq n}\tau_k^{\gamma},\max_{1\leq k\leq n}\tau_k^{n+1}\right)\norm{\sum_{k=1}^{n}\int_{t_{k-1}}^{t_k}|K(t_n-s)|\, ds}_{\infty}& \\ 
=C_1\max\left(\max_{1\leq k\leq n}\tau_k^{\gamma},\max_{1\leq k\leq n}\tau_k^{n+1}\right)\norm{\int_{0}^{t_n}|K(t_n-s)|\, ds}_{\infty}&\\ 
\leq C\max\left(\max_{1\leq k\leq n}\tau_k^{\gamma},\max_{1\leq k\leq n}\tau_k^{n+1}\right),
\end{align}
where again $\displaystyle{C_1 = C_0\max_{0\leq t \leq t_n}\lvert f^{(\gamma)}(s)\rvert}<\infty$ for some $C_0>0$ and $$0<  C_1\norm{K(t_n-s)}_{L^1[0,t_n]} = C < \infty.$$ 
\end{proof}
\end{theorem}
The estimates provided above suggest that the $O(\tau^{n+1})$ term is not attainable directly from the fractional Taylor series expansion. We present an example demonstrating that the kernel $K$ improves this estimate accordingly. 
\begin{example}
Let $K(t) = t^{\alpha-1}$ for $0<\alpha < 1$ and consider an order $\alpha$ approximation to $f(s)$ defined . Then, we define:
\begin{align}
&\lvert R_n\rvert := \int_0^{t_n} (t_n-s)^{\alpha-1}\lvert f(s)-f(t_k)\rvert \, ds \\
&= \sum_{k=1}^{n} \int_{t_{k-1}}^{t_k} \left|\dfrac{(s-t_{k-1})^{\alpha}-\tau_k^{\alpha}}{\Gamma(\alpha+1)} + O\left(f^{(2\alpha)}\right) \right|(t_n-s)^{\alpha-1}\, ds\\
&\leq  \sum_{k=1}^{n} \int_{t_{k-1}}^{t_k} \dfrac{\tau_k^{\alpha}(t_n-s)^{\alpha-1}}{\Gamma(\alpha+1)}\, ds\\
&\leq \dfrac{\tau_n^{\alpha}}{\Gamma(\alpha+1)} \max_{1\leq k \leq n}\tau_k^{\alpha}\\
&=C\tau^{2\alpha},
\end{align}
which is attained under a uniform mesh size $\tau_k = \tau$. However, if $2\alpha >1$, we would directly contradict the usual Taylor series expansion, therefore we obtain the secondary estimate of $C\tau$, since then it is the maximum of that and $C\tau^{2\alpha}$. 
\end{example}
\subsection{Implementation}
Based on the convergence results, we now present a discretization that uses the scheme \eqref{scheme2}. Consider the Volterra equation of the second kind
\begin{align}
u(t) &= f(t) +\int_0^t u(s)K(t-s)\, ds\\
u(0) &= 0,\quad t\in [0,T],
\end{align}
which is to be discretized. By applying \eqref{scheme2}, we then are numerically solving for $u(t_n)$ at each $t_n$ in a uniform mesh where $t_n \in [0,T]$. That is, we wish to solve
\begin{align}
u(t_n) &= f(t_n) +\sum_{k=0}^{n}\tilde{w}_k^{\gamma} u(t_k) + O(\epsilon)\\
u(t_0) &= 0,\quad \forall\, t_n\in [0,T].
\end{align}
To solve for each $u(t_n)$ then, after omitting the $O(\epsilon)$ term and by combining the $u(t_n)$ terms, we can rewrite the equation as
\begin{align}
(1-\tilde{w}_n^{\gamma})u(t_n) &\approx f(t_n) +\sum_{k=0}^{n-1}\tilde{w}_k^{\gamma} u(t_k)\\
u(t_n) &\approx (1-\tilde{w}_n^{\gamma})^{-1}\left(f(t_n) +\sum_{k=0}^{n-1}\tilde{w}_k^{\gamma} u(t_k)\right).
\end{align}
which is invertible provided that $\tilde{w}_n^{\gamma} \neq 1$. Our numerical examples will preserve this condition by selecting various values of the parameter $\alpha$, which is often associated with the order of fractional derivative or integral.
\section{Numerical Examples}
We present three main numerical examples that illustrate our findings: one equation has a kernel that is not weakly singular, another equation has a kernel that is weakly singular, and the last one has a weakly singular kernel function with minimal regularity of the solution. Our first example is a Volterra equation of the second kind with kernel $K(t) = t^{\alpha}$
\begin{align}
u(t) &= f(t) +\int_0^t u(s)(t-s)^{\alpha}\, ds\\
u(0) &= 0,\quad \forall\, t\in [0,T],
\end{align}
with the exact solution $u(t) = t^3$. We remark that the choice of kernel makes the integral a conformal integral equation as presented in section 1. We define $N$ to be the number of uniform partitions of the time domain, $E_{1,\infty}(N)$ to be the maximum error attained over the total mesh for the first example, and $\text{rate}_{1} =\log_2\left(\dfrac{E_{1,\infty}(N/2)}{E_{1,\infty}(N)}\right)$. For this first example, we will take $\alpha = 0.05, 0.25, 0.5, 0.75, 0.95$. Therefore, when $T=1$, the numerical results for a mesh size up to $N=160$ are as follows for the third order accurate scheme:
\\
\\
\renewcommand{\arraystretch}{0.35}
\begin{tabular}{ |p{3cm}||p{3cm}|p{3cm}|p{3cm}|  }
 \hline
 \multicolumn{4}{|c|}{Numerical Error for $u(t)=t^3$, T=1 on a Uniform mesh} \\
 \hline
 $\alpha$ & N & $E_{1,\infty}(N)$ & $\text{rate}_{1}$\\
 \hline
  $0.05$   & 10    &$0.0004$&  *\\
 &20 &$4.7138$e--5 &$2.9761$ \\
 &40 &$5.9554$e--6 &$2.9846$ \\
 &80 &$7.4889$e--7 &$2.9914$ \\
 &160 &$9.3908$e--8	&$2.9954$ \\
 \hline
 $0.25$   & 10    &$0.0003$&  *\\
 &20 &$3.3439$e--5 &$2.9717$ \\
 &40 &$4.2312$e--6 &$2.9824$ \\
 &80 &$5.3247$e--7 &$2.9903$ \\
 &160 &$6.6793$e--8	&$2.9949$ \\
 \hline
 $0.5$   & 10    &$0.0002$&  *\\
 &20 &$2.4606$e--5 &$2.9642$ \\
 &40 &$3.1211$e--6 &$2.9789$ \\
 &80 &$3.9318$e--7 &$2.9888$ \\
 &160 &$4.9345$e--8	&$2.9942$ \\
 \hline
 $0.75$   & 10    &$0.0002$&  *\\
 &20 &$1.9572$e--5 &$2.9566$ \\
 &40 &$2.4882$e--6 &$2.9756$ \\
 &80 &$3.1377$e--7 &$2.9873$ \\
 &160 &$3.9397$e--8	&$2.9936$ \\

 \hline
 $0.95$   & 10    &$0.0001$&  *\\
 &20 &$1.688$e--5 &$2.9509$ \\
 &40 &$2.1497$e--6 &$2.9731$ \\
 &80 &$2.713$e--7 &$2.9862$ \\
 &160 &$3.4077$e--8	&$2.993$ \\

 \hline
\end{tabular}\\
\\
By instead applying the fourth order accurate scheme to the conformable integral equation, we have the following results:
\\
\\
\renewcommand{\arraystretch}{0.35}
\begin{tabular}{ |p{3cm}||p{3cm}|p{3cm}|p{3cm}|  }
 \hline
 \multicolumn{4}{|c|}{Numerical Error for $u(t)=t^3$, T=1 on a Uniform mesh} \\
 \hline
 $\alpha$ & N & $E_{1,\infty}(N)$ & $\text{rate}_{1}$\\
 \hline
  $0.05$   & 10    &$5.6542$e--5&  *\\
 &20 &$3.3089$e--6 &$4.0949$ \\
 &40 &$1.9682$e--7 &$4.0714$ \\
 &80 &$1.1798$e--8 &$4.0603$ \\
 &160 &$7.0981$e--10	&$4.055$ \\
 \hline
 $0.25$   & 10    &$3.8191$e--5&  *\\
 &20 &$1.9751$e--6 &$4.2733$ \\
 &40 &$1.1945$e--7 &$4.0475$ \\
 &80 &$7.4193$e--9 &$4.0089$ \\
 &160 &$4.8019$e--10	&$3.9496$ \\
 \hline
 $0.5$   & 10    &$2.4519$e--5&  *\\
 &20 &$1.51$e--6 &$4.0213$ \\
 &40 &$9.4203$e--8 &$4.0026$ \\
 &80 &$5.8994$e--9 &$3.9971$ \\
 &160 &$3.7104$e--10	&$3.9909$ \\
 \hline
 $0.75$   & 10    &$2.059$e--5&  *\\
 &20 &$1.2952$e--6 &$3.9907$ \\
 &40 &$8.174$e--8 &$3.986$ \\
 &80 &$5.1476$e--9 &$3.9891$ \\
 &160 &$3.3401$e--10	&$3.9459$ \\

 \hline
 $0.95$   & 10    &$1.8683$e--5&  *\\
 &20 &$1.1934$e--6 &$3.9686$ \\
 &40 &$7.586$e--8 &$3.9756$ \\
 &80 &$4.7906$e--9 &$3.9851$ \\
 &160 &$2.9573$e--10	&$4.0178$ \\

 \hline
\end{tabular}\\
\\
Our second example is a Volterra equation of the second kind with kernel $K(t) = t^{\alpha-1}$
\begin{align}
u(t) &= f(t) +\int_0^t u(s)(t-s)^{\alpha-1}\, ds\\
u(0) &= 0,\quad \forall\, t\in [0,T],
\end{align}
which has the exact solution $u(t)=t^6$. This integral equation has an integral term of the form \eqref{CapFI}, and is studied in \cite{DavisX} as a part of a partial differential equation with a fractional integral term. For this second example, we will take $\alpha = 0.1, 0.4, 0.5, 0.7, 0.9$ to ensure that the invertible criterion $\tilde{w}_n^{\gamma} \neq 1$. By using similar definitions for this second example, we have the following results by using the third order accurate scheme:
\\
\\
\renewcommand{\arraystretch}{0.35}
\begin{tabular}{ |p{3cm}||p{3cm}|p{3cm}|p{3cm}|  }
 \hline
 \multicolumn{4}{|c|}{Numerical Error for $u(t)=t^6$, T=1 on a Uniform mesh} \\
 \hline
 $\alpha$ & N & $E_{2,\infty}(N)$ & $\text{rate}_{2}$\\
 \hline
  $0.1$   & 10    &$0.0016$&  *\\
 &20 &$0.0002$ &$2.6934$ \\
 &40 &$3.5321$e--5 &$2.7882$ \\
 &80 &$4.9271$e--6 &$2.8417$ \\
 &160 &$6.719$e--7 &$2.8744$ \\
 \hline
 $0.4$   & 10    &$0.1744$&  *\\
 &20 &$0.02440$ &$2.8377$ \\
 &40 &$0.0034$ &$2.831$ \\
 &80 &$0.0005$ &$2.8889$ \\
 &160 &$6.0717$e--5	&$2.9308$ \\
\hline
 $0.5$   & 10    &$0.0142$&  *\\
 &20 &$0.0020$ &$2.8014$ \\
 &40 &$0.0003$ &$2.8873$ \\
 &80 &$3.5853$e--5 &$2.9366$ \\
 &160 &$4.5967$e--6	&$2.9634$ \\
 \hline
 $0.7$   & 10    &$0.0036$&  *\\
 &20 &$0.0005$ &$2.8581$ \\
 &40 &$6.4879$e--5 &$2.9281$ \\
 &80 &$8.3172$e--6 &$2.9636$ \\
 &160 &$1.0532$e--6	&$2.9814$ \\
 \hline
 $0.9$   & 10    &$0.0018$&  *\\
 &20 &$0.0002$ &$2.8829$ \\
 &40 &$3.0942$e--5 &$2.9418$ \\
 &80 &$3.9459$e--6 &$2.9711$ \\
 &160 &$4.9816$e--7 &$2.9857$ \\
 \hline
\end{tabular}\\
\\
By applying the fourth order scheme to the second example, we have the following results:
\\
\\
\renewcommand{\arraystretch}{0.35}
\begin{tabular}{ |p{3cm}||p{3cm}|p{3cm}|p{3cm}|  }
 \hline
 \multicolumn{4}{|c|}{Numerical Error for $u(t)=t^6$, T=1 on a Uniform mesh} \\
 \hline
 $\alpha$ & N & $E_{2,\infty}(N)$ & $\text{rate}_{2}$\\
\hline
 $0.1$   & 10    &$0.0003$&  *\\
 &20 &$2.4748$e--5 &$3.6826$ \\
 &40 &$1.7872$e--6 &$3.7916$ \\
 &80 &$1.2397$e--7 &$3.8496$ \\
 &160 &$8.3889$e--9 &$3.8854$ \\
\hline
 $0.4$   & 10    &$0.0593$&  *\\
 &20 &$0.0051$ &$3.54$ \\
 &40 &$0.0004$ &$3.7111$ \\
 &80 &$2.7211$e--5 &$3.8375$ \\
 &160 &$1.8124$e--6	&$3.9083$ \\
 \hline
 $0.5$   & 10    &$0.0038$&  *\\
 &20 &$0.0003$ &$3.6983$ \\
 &40 &$2.0251$e--5 &$3.8425$ \\
 &80 &$1.3408$e--6 &$3.9168$ \\
 &160 &$8.6413$e--8	&$3.9557$ \\
 \hline
 $0.7$   & 10    &$0.0009$&  *\\
 &20 &$6.2452$e--5 &$3.7849$ \\
 &40 &$4.1996$e--6 &$3.8944$ \\
 &80 &$2.7217$e--7 &$3.9477$ \\
 &160 &$1.7289$e--8	&$3.9766$ \\
 \hline
  $0.9$   & 10    &$0.0004$&  *\\
 &20 &$2.9927$e--5 &$3.8057$ \\
 &40 &$1.9981$e--6 &$3.9048$ \\
 &80 &$1.2901$e--7 &$3.953$ \\
 &160 &$8.175$e--9	&$3.9802$ \\
 \hline

\end{tabular}\\
\\
We remark that while the selection of $\alpha$ is to preserve the invertibility condition, when we have $\alpha=0.25$, we have spurious and large blowup for small values of $N$, but as $N\rightarrow \infty$, we still exhibit the appropriate order of convergence, and hence still preserve the stability condition. For example, using the fourth order scheme for the second example, with $\alpha = 0.25$, we have the following rate of convergence for up to $N=5120$:
\\
\\
\renewcommand{\arraystretch}{0.35}
\begin{tabular}{ |p{3cm}||p{3cm}|p{3cm}|}
 \hline
 \multicolumn{3}{|c|}{Numerical rate for $u(t)=t^6$, T=1 on a Uniform mesh} \\
 \hline
 $\alpha$ & N & $\text{rate}_{2}$\\
 \hline
  $0.25$   & 10    &  *\\
 &20  &$-3.5302$ \\
 &40  &$-55.741$ \\
 &80  &$-295.64$ \\
 &160 &$126.68$ \\
 &320 &$11.106$ \\
 &640 &$4.3374$ \\
 &1280 &$3.7954$ \\
 &2560 &$3.8389$ \\
 &5120 &$3.8916$ \\
 &10240 &$3.945$ \\
 \hline
\end{tabular}\\
\\
Our third example is the Volterra equation of the second kind that is motivated by the findings in \cite{Zhang2014} and \cite{DavisX}. This particular equation is obtained by applying the Laplace transform to equation (1.2) of \cite{Zhang2014} to obtain:
\begin{align}
u(x,t) &= \phi(x)+ \int_0^t \left( f(x,s) + \dfrac{\partial^2u}{\partial x^2}(x,s)\right)\dfrac{(t-s)^{\alpha-1}}{\Gamma(\alpha)}\, ds\\
u(x,0) &= \phi(x),\quad u\vert_{\partial \Omega} =0,
\end{align}
on the interval $x\in [0,1]$, $t\in [0,1]$, which has the initial condition $\phi(x) = 0$ and the exact solution $u(x,t) =\sin(\pi x) t^{\alpha}$. We apply a fixed fourth order Laplacian operator in space as in \cite{Zhang2014} and the $\alpha$ order scheme presented in Section 3 to analyze the problem. We now define $E_{\alpha,\infty}(N)$ and $\text{rate}_{\alpha}$ analogously to the previous examples. By fixing $M=25$ space steps and consider $\alpha = 0.05,0.25,0.5,0.75, 0.75, 0.95$, we have the following numerical results:
\\
\\
\renewcommand{\arraystretch}{0.35}
\begin{tabular}{ |p{3cm}||p{3cm}|p{3cm}|p{3cm}|  }
 \hline
 \multicolumn{4}{|c|}{Numerical Error for $u(x,t)=\sin(\pi x)t^{\alpha}$, T=1 on a Uniform mesh} \\
 \hline
 $\alpha$ & N & $E_{\alpha,\infty}(N)$ & $\text{rate}_{\alpha}$\\
\hline
 $0.05$   & 10    &$0.0260$&  *\\
 &20 &$0.0250$ &$0.0551$ \\
 &40 &$0.0240$ &$0.0552$ \\
 &80 &$0.0231$ &$0.05539$ \\
 &160 &$0.0223$ &$0.0556$ \\
\hline
 $0.25$   & 10    &$0.0588$&  *\\
 &20 &$0.0481$ &$0.2878$ \\
 &40 &$0.0393$ &$0.2937$ \\
 &80 &$0.0319$ &$0.3003$ \\
 &160 &$0.0258$	&$0.3077$ \\
 \hline
 $0.5$   & 10    &$0.0528$&  *\\
 &20 &$0.0003$ &$0.6265$ \\
 &40 &$0.0216$ &$0.6618$ \\
 &80 &$0.0133$ &$0.7017$ \\
 &160 &$0.0079$	&$0.7443$ \\
 \hline
 $0.75$   & 10    &$0.0425$&  *\\
 &20 &$0.0208$ &$1.0331$ \\
 &40 &$0.0114$ &$0.8618$ \\
 &80 &$0.0062$ &$0.8904$ \\
 &160 &$0.0033$	&$0.9178$ \\
 \hline
  $0.95$   & 10    &$0.0470$&  *\\
 &20 &$0.0238$ &$0.9792$ \\
 &40 &$0.0120$ &$0.9863$ \\
 &80 &$0.0061$ &$0.9916$ \\
 &160 &$0.0030$ &$0.9953$ \\
 \hline
\end{tabular}\\
\\
Of particular interest is the cases where $\alpha \geq 1/2$, which validate the findings in Example 3.8. If we instead consider our exact solution to be $u(x,t) = \sin(\pi x) t^{1-\alpha}$ and apply the same $\alpha$ order in time scheme, we exhibit an even better convergence for the smaller values of $\alpha$ than predicted:
\\
\\
\renewcommand{\arraystretch}{0.35}
\begin{tabular}{ |p{3cm}||p{3cm}|p{3cm}|p{3cm}|  }
 \hline
 \multicolumn{4}{|c|}{Numerical Error for $u(x,t)=\sin(\pi x)t^{1-\alpha}$, T=1 on a Uniform mesh} \\
 \hline
 $\alpha$ & N & $E_{\alpha,\infty}(N)$ & $\text{rate}_{\alpha}$\\
\hline
 $0.05$   & 10    &$0.0079$&  *\\
 &20 &$0.0046$ &$0.7577$ \\
 &40 &$0.0026$ &$0.8154$ \\
 &80 &$0.0015$ &$0.8486$ \\
 &160 &$0.0008$ &$0.8705$ \\
\hline
 $0.25$   & 10    &$0.0255$&  *\\
 &20 &$0.0148$ &$0.7878$ \\
 &40 &$0.0085$ &$0.7937$ \\
 &80 &$0.0049$ &$0.8003$ \\
 &160 &$0.0028$	&$0.8077$ \\
 \hline
 $0.5$   & 10    &$0.0528$&  *\\
 &20 &$0.0003$ &$0.6265$ \\
 &40 &$0.0216$ &$0.6618$ \\
 &80 &$0.0133$ &$0.7017$ \\
 &160 &$0.0079$	&$0.7443$ \\
 \hline
 $0.75$   & 10    &$0.176$&  *\\
 &20 &$0.1199$ &$0.5537$ \\
 &40 &$0.0764$ &$0.6497$ \\
 &80 &$0.0457$ &$0.7426$ \\
 &160 &$0.0258$	&$0.8219$ \\
 \hline
  $0.95$   & 10    &$0.4123$&  *\\
 &20 &$0.2773$ &$0.5735$ \\
 &40 &$0.1687$ &$0.7172$ \\
 &80 &$0.0950$ &$0.8284$ \\
 &160 &$0.0508$ &$0.9025$ \\
 \hline
\end{tabular}\\
\\

\end{document}